\documentclass[11pt,a4paper]{amsart}
\usepackage{amssymb}
\usepackage{graphicx}
\usepackage[colorlinks,
linkcolor=blue,
anchorcolor=blue,
citecolor=blue]{hyperref}

\theoremstyle{plain}
\newtheorem{Main}{Theorem}

\newtheorem{Thm}{Theorem}[section]
\newtheorem{Prop}[Thm]{Proposition}
\newtheorem{Lem}[Thm]{Lemma}

\newtheorem{Cor}[Thm]{Corollary}

\theoremstyle{remark}

\newtheorem{Def}[Thm]{Definition}

\def\max{\operatorname{max}}

\begin{document}
	
	\title[non-additive neutralized Bowen topological pressure]
	{Variational principle for non-additive neutralized Bowen topological pressure}
	\author{Congcong Qu}
	\address{Congcong Qu, College of Big Data and software Engineering, Zhejiang Wanli University, Ningbo, 315107, Zhejiang, P.R.China}
	\email{congcongqu@foxmail.com}
	
	\author{Lan Xu \textsuperscript{*}}
	\address{Lan Xu, Department of Mathematics and Physics, Suzhou Vocational University, Suzhou 215104, Jiangsu, P.R.China}
	\email{xlan@jssvc.edu.cn}

	\date{\today}

	\thanks{\textsuperscript{*}Corresponding author}
	
	

	\begin{abstract}
		Ovadia and Rodriguez-Hertz \cite{OH} defined the neutralized Bowen open ball as 
		$$B_n(x,e^{-n\varepsilon})=\{y\in X:d(T^j(x),T^j(y))<e^{-n\varepsilon},\forall 0\leq j\leq n-1\}.$$
		Yang, Chen and Zhou \cite{YCZ} introduced the notion of neutralized Bowen topological entropy of subsets by replacing the usual Bowen ball by neutralized Bowen open ball. And they established variational principles for this notion. In this note, we extend this result to the non-additive neutralized Bowen topological pressure and established the variational principle for non-additive potentials with tempered distortion.
	\end{abstract}
	
	\maketitle
	
	
	\section{Introduction}
	Given a compact metric space $X$ with a metric $d$ and a continuous self-map $T$ on $X$, we say $(X,T)$ is a topological dynamical system (TDS for short). Denote the set of all Borel probability measures on $X$ by $M(X)$.
	
	The topological entropy was introduced by Adler, Konheim and McAndrew \cite{AKM} to describe the topological complexity of dynamical systems. Later, Bowen \cite{Bow71} gave equivalent definitions for topological entropy  by spanning set and separated set. Besides, he \cite{Bow73} gave a dimensional description of entropy, which led to fruitful results in dimension theory, ergodic theory, multifractal analysis and other fields of dynamical systems. The topological entropy  was generalized to topological pressure for $\mathbb{Z}-$action on a compact set by Ruelle \cite{R} for expansive maps and in the general case by Walters \cite{Wal76}. Also Ruelle \cite{R} formulated a variational principle for topological pressure and Walters \cite{Wal76} generalized this to general continuous maps on the compact metric space. Pesin and Pitskel' \cite{PP84} generalized this to arbitrary subsets and established a variational principle for arbitrary subsets. Barreira \cite{Ba96} introduced a non-additive topological pressure and established a non-additive version of the classical variational principle. Cao, Feng and Huang \cite{CFH} gave the notion of the sub-additive topological pressure and formulated a variational principle for this. Feng and Huang established a variational principle of the topological entropy for non-empty compact set, not necessarily invariant. Tang, Cheng and Zhao generalized this to topological pressure. 
	
	Recently, Ovadia and Rodriguez-Hertz \cite{OH} defined the neutralized Bowen open ball as 
	$$B_n(x,e^{-n\varepsilon})=\{y\in X:d(T^j(x),T^j(y))<e^{-n\varepsilon},\forall 0\leq j\leq n-1\}.$$
	Yang, Chen and Zhou \cite{YCZ} introduced the notion of neutralized Bowen topological entropy of subsets by replacing the usual Bowen ball by neutralized Bowen open ball. And they established variational principles for this notion. Sarkooh, Ehsani, Pashaei and Abdi \cite{SEPA} generalized this to non-autonomous dynamical systems.
	
	In this note, we extend Yang, Chen and Zhou 's result \cite{YCZ} to non-additive neutralized Bowen topological pressure and establish a variational principle for non-additive potentials with tempered distortion. This includes the classical additive ones and non-additive ones with bounded distortion, especially the sub-additive topological pressure of the sub-additive potentials with bounded distortion. 
	
	Given a TDS $(X,T)$ and a sequence of continuous potentials $\varPhi=\{\varphi_n: X\rightarrow \mathbb{R}\}$. Denote by $P^{\tilde{B}}_Z(T,\varPhi)$ the non-additive neutralized Bowen topological pressure of $T$ with respect to $\varPhi$ on the set $Z$, $\underline{P}_{\mu}^{\widetilde{BK}}(T,\varPhi,\varepsilon)$ the lower neutralized Brin-Katok local pressure of the Borel probability measure $\mu$ and $P_{\mu}^{\widetilde{K}}(T,\varPhi,\varepsilon)$ the neutralized Katok pressure of $\mu$. See Section \ref{Preliminary} for details. The main result in this note is the following.  
	\begin{Main}
		Let $(X,T)$ be a TDS, $\varPhi=\{\varphi_n: X\rightarrow \mathbb{R}\}$ be a sequence of continuous potentials and $Z$ be a non-empty compact subset of $X$. We assume that $\{\varphi_n\}_{n\in\mathbb{N}}$ satisfy the tempered distortion condition 
		\begin{align}\label{tempered}
			\lim_{\varepsilon\rightarrow 0}\limsup_{n\rightarrow \infty}\frac{\varphi_n(\varepsilon)}{n}=0,
		\end{align}
		where $\varphi_n(\varepsilon)=\sup\{|\varphi_n(x)-\varphi_n(y)|:x\in X, y\in B_n(x,e^{-n\varepsilon})\}.$
		Then 
		\begin{align*}
			P^{\tilde{B}}_Z(T,\varPhi)&=\lim_{\varepsilon\rightarrow 0}\sup\{\underline{P}_{\mu}^{\widetilde{BK}}(T,\varPhi,\varepsilon):\mu \in M(X),\, \mu(Z)=1\}\\
			&=\lim_{\varepsilon\rightarrow 0}\sup\{P_{\mu}^{\widetilde{K}}(T,\varPhi,\varepsilon):\mu \in M(X),\, \mu(Z)=1\}.
		\end{align*}
	\end{Main}
	This paper is organized as follows. In section \ref{Preliminary}, we recall some basic definitions. In section \ref{proof}, we give the details of the proof of the main result. 
	\section{Preliminary}\label{Preliminary}
	In this section, we introduce some notions to appear in this note, which generalize the corresponding ones in \cite{YCZ}.
	
	\subsection{Non-additive neutralized Bowen topological pressure}Given $n\in \mathbb{N}$, $x,y\in X$, the $n-$th Bowen metric $d_n$ is defined as 
	$$d_n(x,y)\triangleq \max_{0\leq j\leq n-1}d(T^j(x),T^j(y)).$$
	The Bowen open ball of radius $\varepsilon$ and order $n$ in the metric $d_n$ around $x$ is given by 
	$$B_n(x,\varepsilon)=\{y\in X:d_n(x,y)<\varepsilon\}.$$ 
	In \cite{OH}, the authors defined the neutralized Bowen open ball by replacing the radius $\varepsilon$ in the usual Bowen ball by $e^{-n\varepsilon}$, that is 
	$$B_n(x,e^{-n\varepsilon})=\{y\in X:d_n(x,y)<e^{-n\varepsilon}\}.$$
	With this, Yang, Chen and Zhou \cite{YCZ} defined the neutralized Bowen topological entropy on subsets. We extend this to the non-additive neutralized Bowen topological pressure as follows.
	\begin{Def}\label{neutralized Bowen tp}
		Let $Z\subset X$ be a non-empty subset and $\varPhi=\{\varphi_n: X\rightarrow \mathbb{R}\}$ be a sequence of continuous potentials, and let $\varepsilon>0$, $N\in\mathbb{N}$, $s\in\mathbb{R}$. Define
		$$M^{s}_{N,\varepsilon}(Z,\varPhi)=\inf\sum_{i\in I} \exp{[-n_i s+\sup_{y\in B_{n_i}(x_i,e^{-n_i\varepsilon})}\varphi_{n_i}(y)]},$$
		where the infimum is taken over all finite or countable covers $\{B_{n_i}(x_i,e^{-n_i\varepsilon})\}_{i\in I}$ of $Z$ with $n_i\geq N$ and $x_i\in X$.
	\end{Def}
	One can check that the limit $M_{\varepsilon}^{s}(Z,\varPhi)=\lim_{N\rightarrow \infty} M_{N,\varepsilon}^{s}(Z,\varPhi)$ exists since $M^{s}_{N,\varepsilon}(Z,\varPhi)$ is non-decreasing when $N$ increases. The quantity $M^s_{\varepsilon}(Z,\varPhi)$ has a critical value of parameter $s$ jumping from $\infty$ to $0$. The critical value is defined by 
	$$M_{\varepsilon}(Z,\varPhi)\triangleq \inf\{s:M^s_{\varepsilon}(Z,\varPhi)=0\}=\sup\{s:M^s_{\varepsilon}(Z,\varPhi)=\infty\}.$$
	The {\it non-additive neutralized Bowen topological pressure of $T$ with respect to $\varPhi$ on the set $Z$} is defined by 
	$$P_Z^{\tilde{B}}(T,\varPhi)=\lim_{\varepsilon\rightarrow 0} M_{\varepsilon}(Z,\varPhi)=\inf_{\varepsilon>0} M_{\varepsilon}(Z,\varPhi).$$
	
	Assume  the potentials $\varPhi$ satisfies the tempered distortion condition (\ref{tempered}). Given $B_{n_i}(x_i,e^{-n_i\varepsilon})$ and $z\in B_{n_i}(x_i,e^{-n_i\varepsilon})$, we can replace $\sup_{y\in B_{n_i}(x_i,e^{-n_i\varepsilon})}\varphi_{n_i}(y)$ by $\varphi_{n_i}(z)$ in Definition \ref{neutralized Bowen tp} to define the non-additive neutralized Bowen topological pressure. We denote by $\tilde{M}^{s}_{N,\varepsilon}(Z,\varPhi)$, $\tilde{M}_{\varepsilon}(Z,\varPhi)$ and $\tilde{P}_Z^{\tilde{B}}(T,\varPhi)$ the new corresponding quantities of $M^{s}_{N,\varepsilon}(Z,\varPhi)$, $M_{\varepsilon}(Z,\varPhi)$ and $P_Z^{\tilde{B}}(T,\varPhi)$ respectively. 
	\begin{Prop}\label{inter}
		Assume $\varPhi$ satisfies the tempered distortion condition (\ref{tempered}). Then $$P_Z^{\tilde{B}}(T,\varPhi)=\lim_{\varepsilon\rightarrow 0}\tilde{M}_{\varepsilon}(Z,\varPhi).$$
	\end{Prop}
	\begin{proof}
		By (\ref{tempered}), for $\delta>0$, there exists $\varepsilon_0>0$ such that for any $0<\varepsilon<\varepsilon_0$, there exists $N_0\in \mathbb{N}$ such that for any $n\geq N_0$, one has
		\begin{align*}
			\varphi_{n}(\varepsilon)\leq n\delta.
		\end{align*}
		Thus for $s\in \mathbb{R}$ and $N\geq N_0$, one has $M_{N,\varepsilon}^{s}(Z,\varPhi)\leq \tilde{M}_{N,\varepsilon}^{s-\delta}(Z,\varPhi)\leq M_{N,\varepsilon}^{s-\delta}(Z,\varPhi)$.
		Letting $N$ tend to $\infty$, we have $M_{\varepsilon}^{s}(Z,\varPhi)\leq \tilde{M}_{\varepsilon}^{s-\delta}(Z,\varPhi)\leq M_{\varepsilon}^{s-\delta}(Z,\varPhi)$. It follows that $M_{\varepsilon}(Z,\varPhi)\leq \tilde{M}_{\varepsilon}(Z,\varPhi)+\delta\leq M_{\varepsilon}(Z,\varPhi)+\delta$. Letting $\varepsilon>0$ tend to $0$, we have $P_Z^{\tilde{B}}(T,\varPhi)\leq \tilde{P}_Z^{\tilde{B}}(T,\varPhi)+\delta\leq P_Z^{\tilde{B}}(T,\varPhi)+\delta$. The result follows from the arbitrariness of $\delta>0$.
	\end{proof}
	\subsection{Lower neutralized Brin-Katok local pressure}
	Given $\varPhi=\{\varphi_n: X\rightarrow \mathbb{R}\}$  a sequence of continuous potentials, $\mu\in M(X)$, $\varepsilon>0$ and $x\in X$, define
	$$\underline{P}_{\mu}^{\widetilde{BK}}(T,\varPhi,\varepsilon,x)=\liminf_{n\rightarrow \infty} \frac{-\log \mu(B_n(x,e^{-n\varepsilon}))+\varphi_{n}(x)}{n},$$
	and
	$$\underline{P}_{\mu}^{\widetilde{BK}}(T,\varPhi,\varepsilon)=\int\underline{P}_{\mu}^{\widetilde{BK}}(T,\varPhi,\varepsilon,x)d\mu.$$
	The {\it lower neutralized Brin-Katok local pressure of $\mu$} is given by
	$$\underline{P}_{\mu}^{\widetilde{BK}}(T,\varPhi)=\lim_{\varepsilon\rightarrow 0}\underline{P}_{\mu}^{\widetilde{BK}}(T,\varPhi,\varepsilon).$$
	
	\subsection{Neutralized Katok's pressure}
	\begin{Def}
		Let $\varepsilon>0$, $s\in\mathbb{R}$, $N\in\mathbb{N}$, $\mu\in M(X)$, $\delta\in (0,1)$ and $\varPhi=\{\varphi_n: X\rightarrow \mathbb{R}\}$ be a sequence of continuous potentials. Put 
		$$\Lambda^s_{N,\varepsilon}(\mu,\varPhi,\delta)=\inf\sum_{i\in I}\exp{[-n_i s+\sup_{y\in B_{n_i}(x_i,e^{-n_i\varepsilon})}\varphi_{n_i}(y)]},$$
		where the infimum is taken over all finite or countable covers $\{B_{n_i}(x_i,e^{-n_i \varepsilon})\}_{i\in I}$ so that $\mu(\bigcup_{i\in I}B_{n_i}(x_i,e^{-n_i\varepsilon}))>1-\delta$ with $n_i\geq N$ and $x_i\in X$.
		
		Let $\Lambda^s_{\varepsilon}(\mu,\varPhi,\delta)=\lim_{N\rightarrow \infty} \Lambda^s_{N,\varepsilon}(\mu,\varPhi,\delta)$. There is a critical value of $s$ for $\Lambda^s_{\varepsilon}(\mu,\varPhi,\delta)$ jumping from $\infty$ to $0$. Define the critical value as 
		$$\Lambda_{\varepsilon}(\mu,\varPhi,\delta)=\inf\{s:\Lambda^s_{\varepsilon}(\mu,\varPhi,\delta)=0\}=\sup\{s:\Lambda^s_{\varepsilon}(\mu,\varPhi,\delta)=\infty\}.$$
		Let $P_{\mu}^{\tilde{K}}(T,\varPhi,\varepsilon)=\lim_{\delta\rightarrow 0} \Lambda_{\varepsilon}(\mu,\varPhi,\delta)$. The {\it neutralized Katok's pressure of $\mu$} is defined as 
		$$P_{\mu}^{\tilde{K}}(T,\varPhi)=\lim_{\varepsilon \rightarrow 0}P_{\mu}^{\tilde{K}}(T,\varPhi,\varepsilon).$$
	\end{Def}
	\section{Proof of the main result}\label{proof}
	\subsection{Non-additive neutralized weighted Bowen topological pressure}
	To give the proof of main result, we define the non-additive neutralized weighted Bowen topological pressure.
	\begin{Def}
		Let $\varPhi=\{\varphi_n: X\rightarrow \mathbb{R}\}$ be a sequence of continuous potentials and $f:X\rightarrow \mathbb{R}$ be a bounded real-valued function on $X$. Let $s\in \mathbb{R},N\in\mathbb{N}$ and $\varepsilon>0$. Define
		$$W^s_{N,\varepsilon}(f,\varPhi)=\inf\sum_{i\in I}c_i \exp{[-n_i s+\sup_{y\in B_{n_i}(x_i,e^{-n_i\varepsilon})}\varphi_{n_i}(y)]},$$
		where the infimum is taken over all finite or countable families $\{(B_{n_i}(x_i,e^{-n_i\varepsilon}),c_i)\}_{i\in I}$ with $0<c_i<\infty$, $x_i \in X$ and $n_i\geq N$ so that 
		$$\sum_{i\in I}c_i \chi_{B_{n_i}(x_i,e^{-n_i \varepsilon})}\geq f,$$ 
		where $\chi_A$ denotes the characteristic function of $A$.
		
		Let $Z\subset X$ be a non-empty subset. Set $W^s_{N,\varepsilon}(Z,\varPhi)\triangleq W^s_{N,\varepsilon}(\chi_Z,\varPhi)$. Let $W^s_{\varepsilon}(Z,\varPhi)=\lim_{N\rightarrow \infty} W^s_{N,\varepsilon}(Z,\varPhi)$. There is a critical value of $s$ such that the quantity $W^s_{\varepsilon}(Z,\varPhi)$ jumps from $\infty$ to $0$. Denote this critical value as 
		$$W_{\varepsilon}(Z,\varPhi)\triangleq\inf\{s:W^s_{\varepsilon}(Z,\varPhi)=0\}=\sup\{s:W^s_{\varepsilon}(Z,\varPhi)=\infty\}.$$
		The {\it non-additive neutralized weighted Bowen topological pressure of $T$ with respect to $\varPhi$ on the set $Z$ }is defined as 
		$$P_Z^{\widetilde{WB}}(T,\varPhi)=\lim_{\varepsilon\rightarrow 0}  W_{\varepsilon}(Z,\varPhi).$$
	\end{Def}

	The following covering lemma helps us to establish the equivalence of the non-additive neutralized Bowen topological pressure and non-additive neutralized weighted Bowen topological pressure.
	\begin{Lem}\cite{Wang}\label{cover}
		Let $(X,d)$ be a compact metric space. Let $r>0$ and $\mathcal{B}=\{B(x_i,r)\}_{i\in I}$ be a family of open balls of $X$. Define 
		\begin{align*}
			I(i)=\{j\in I:B(x_j,r)\cap B(x_i,r)\not=\varnothing\}.
		\end{align*}
		Then there exists a finite subset $J\subset I$ such that for any $i,j\in J$ with $i\not =j$, $I(i)\cap I(j)=\varnothing$ and 
		$$\bigcup_{i\in I}B(x_i,r)\subset \bigcup_{j\in J}B(x_j,5r).$$
	\end{Lem}
	Given a non-empty subset $Z\subset X$ and $\varepsilon>0$, $N\in\mathbb{N}$, $s\in\mathbb{R}$. Define
	$$\mathcal{M}^{s}_{N,\varepsilon}(Z,\varPhi)=\inf\sum_{i\in I} \exp{[-n_i s+\varphi_{n_i}(x_i)]},$$
	where the infimum is taken over all finite or countable covers $\{B_{n_i}(x_i,e^{-n_i\varepsilon})\}_{i\in I}$ of $Z$ with $n_i\geq N$ and $x_i\in X$. 
	
	The proof of the following result is inspired by \cite{YCZ,TCZ}.
\begin{Prop}\label{equality}
		Let $Z$ be a non-empty subset of $X$, $\varPhi=\{\varphi_n: X\rightarrow \mathbb{R}\}$ be a sequence of continuous potentials satisfying the tempered distortion condition (\ref{tempered}). Then for  $\theta>0$, there exists $\varepsilon_0>0$ such that for $0<\varepsilon<\varepsilon_0$, there exists $N_0\in\mathbb{N}$, such that for any $N\geq N_0$ and  $s\in \mathbb{R}$, 
	$$\mathcal{M}^{s+\theta}_{N,\frac{\varepsilon}{2}}(Z,\varPhi)\leq W^s_{N,\varepsilon}(Z,\varPhi)\leq M^s_{N,\varepsilon}(Z,\varPhi).$$
 Therefore, we have $P_{Z}^{\tilde{B}}(T,\varPhi)=P_{Z}^{\widetilde{WB}}(T,\varPhi)$. 
\end{Prop}
\begin{proof}
	The inequality $W^s_{N,\varepsilon}(Z,\varPhi)\leq M^s_{N,\varepsilon}(Z,\varPhi)$ follows from definition. Next we show that $\mathcal{M}^{s+\theta}_{N,\frac{\varepsilon}{2}}(Z,\varPhi)\leq W^{s}_{N,\varepsilon}(Z,\varPhi)$. By (\ref{tempered}), for $\theta>0$, there exists $\varepsilon_0>0$ such that for any $0<\varepsilon<\varepsilon_0$, there exists $N_0\in \mathbb{N}$ such that for any $n\geq N_0$, one has
	\begin{align*}
		\varphi_{n}(\varepsilon)\leq \frac{n\theta}{2}.
	\end{align*}
Fix $N\geq N_0$ large enough such that $e^{\frac{n\varepsilon}{2}}>5$ and $\frac{n^2}{e^{\frac{n\theta}{2}}}<1$ for all $n\geq N$. Let $t>0$ and $n\geq N$. Let $\{(B_{n_i}(x_i,e^{-n_i\varepsilon}),c_i)\}_{i\in I}$ with $0<c_i<\infty, x_i\in X, n_i\geq N$ be a finite or countable family satisfying $\sum_{i\in I} c_i\chi_{B_{n_i}(x_i,e^{-n_i\varepsilon})}\geq \chi_{Z}$. Define $I_n=\{i\in I:n_i=n\}$ for $n\geq N$. Define 
	\begin{align}
		Z_{n,t}=\{z\in Z:\sum_{i\in I_n}c_i\chi_{B_n(x_i,e^{-n\varepsilon})}(z)>t\}
	\end{align}
	and
	\begin{align}
		I_n^t=\{i\in I_n:B_n(x_i,e^{-n\varepsilon})\cap Z_{n,t}\not=\varnothing\}.
	\end{align}
	It follows that $Z_{n,t}\subset \bigcup_{i\in I_n^t} B_n(x_i,e^{-n\varepsilon})$. Let $\mathcal{B}=\{B_n(x_i,e^{-n\varepsilon})\}_{i\in I_n^t}$. By Lemma \ref{cover}, there exists a finite subset $J\subset I^t_n$ such that
	$$\bigcup_{i\in I_n^t} B_n(x_i,e^{-n\varepsilon})\subset \bigcup_{j\in J} B_n(x_j,5e^{-n\varepsilon})\subset \bigcup_{j\in J} B_n(x_j,e^{-\frac{n\varepsilon}{2}})$$
	and for any $i,j\in J$ with $i\not= j$, one has $I_n^t(i)\cap I_n^t(j)=\varnothing$, where $$I_n^t(i)=\{j\in I_n^t:B_n(x_j,e^{-n\varepsilon})\cap B_n(x_i,e^{-n\varepsilon})\not=\varnothing\}.$$ For each $j\in J$, choose $y_j\in B_n(x_j,e^{-n\varepsilon})\cap Z_{n,t}$. Then $\sum_{i\in I_n^t}c_i\chi_{B_n(x_i,e^{-n\varepsilon})}(y_j)>t$ and hence $\sum_{i\in I_n^t(j)} c_i>t$. For each $j\in J$ and $i\in I_n^t(j)$, there exists $z_i\in B_n(x_j,e^{-n\varepsilon})\cap B_n(x_i,e^{-n\varepsilon})$. Thus for each $j\in J$, $$\sum_{i\in I_n^t(j)} c_i\exp{\varphi_{n}(z_i)}\geq\sum_{i\in I_n^t(j)} c_i\exp{[\varphi_{n}(x_j)-\frac{n\theta}{2}]}>t\exp{[\varphi_n(x_j)-\frac{n\theta}{2}]}.$$It follows that 
	\begin{align*}
		\sum_{j\in J}\exp{[\varphi_{n}(x_j)-\frac{n\theta}{2}]}&<\frac{1}{t}\sum_{j\in J}\sum_{i\in I_n^t(j)} c_i\exp{\varphi_{n}(z_i)}\\
		&\leq \frac{1}{t}\sum_{i\in I_n^t} c_i\exp{\sup_{y\in B_{n}(x_i,e^{-n\varepsilon})}\varphi_{n}(y)}.
	\end{align*}
	Therefore,
	\begin{align*}
		\mathcal{M}^{s+\theta}_{N,\frac{\varepsilon}{2}}(Z_{n,t},\varPhi)&\leq \sum_{j\in J} \exp{[-n (s+\theta)+\varphi_{n}(x_j)]}\\&
		\leq \frac{1}{n^2 t}\sum_{i \in I_n} c_i \exp{[-n s+\sup_{y\in B_{n}(x_i,e^{-n\varepsilon})}\varphi_{n}(y)]}.  
	\end{align*}
	By definition, $Z=\bigcup_{n\geq N} Z_{n,\frac{1}{n^2}t}$. Hence
	$$\mathcal{M}_{N,\frac{\varepsilon}{2}}^{s+\theta}(Z,\varPhi)\leq \sum_{n\geq N} \mathcal{M}_{N,\frac{\varepsilon}{2}}^{s+\theta}(Z_{n,\frac{1}{n^2}t},\varPhi)\leq \frac{1}{t}\sum_{i\in I}c_i \exp{[-n_i s+\sup_{y\in B_{n_i}(x_i,e^{-n_i\varepsilon})}\varphi_{n_i}(y)]}.$$
	Let $t\rightarrow 1$. It follows that $\mathcal{M}^{s+\theta}_{N,\frac{\varepsilon}{2}}(Z,\varPhi)\leq W^{s}_{N,\varepsilon}(Z,\varPhi)$. Letting $N$ tend to $\infty$, we have $\mathcal{M}^{s+\theta}_{\frac{\varepsilon}{2}}(Z,\varPhi)\leq W^s_{\varepsilon}(Z,\varPhi)\leq M^s_{\varepsilon}(Z,\varPhi).$  It follows that $\mathcal{M}_{\frac{\varepsilon}{2}}(Z,\varPhi)\leq W_{\varepsilon}(Z,\varPhi)+\theta\leq M_{\varepsilon}(Z,\varPhi).$ Let $\varepsilon$ tend to $0$. The result follows from the arbitrariness of $\theta>0$ and  Proposition \ref{inter}.
\end{proof}
	\subsection{Proof of the Main Theorem}
	\begin{Lem}\label{Frostman}
		(Frostman's lemma). Let $Z$ be a non-empty compact subset of $X$ and $s\in \mathbb{R}$, $N\in\mathbb{N}$ and $\varepsilon>0$. Set $c=W^s_{N,\varepsilon}(Z,\varPhi)>0$. Then there exists a Borel probability measure $\mu$ on $X$ such that $\mu(Z)=1$ and for any $x\in X$, $n\geq N$, 
		$$\mu(B_n(x,e^{-n\varepsilon}))\leq \frac{1}{c}\exp{[-ns+\sup_{y\in B_{n}(x,e^{-n\varepsilon})}\varphi_{n}(y)]}.$$
	\end{Lem}
	\begin{proof}
		The proof of \cite{FH} still works by replacing the ball $B_n(x,\varepsilon)$ by $B_n(x,e^{-n\varepsilon})$ and adding a potential function. So we omit the detail.
	\end{proof}
	The proof of  the main result is inspired by \cite{TCZ,YCZ,ZC}.
	\begin{proof}[Proof of the main theorem]
		For every $\mu\in M(X)$ with $\mu(Z)=1$ and $\varepsilon>0$, by definition we have 
		\begin{align}\label{less}
			P_{\mu}^{\widetilde{K}}(T,\varPhi,\varepsilon)\leq M_{\varepsilon}(Z,\varPhi). 
		\end{align}
		By definition, 
		$$\underline{P}_{\mu}^{\widetilde{BK}}(T,\varPhi,\varepsilon)=\int \underline{P}_{\mu}^{\widetilde{BK}}(T,\varPhi,\varepsilon,x)d\mu.$$
		It follows that 
		$$Z_\eta=\{x\in Z:\underline{P}_{\mu}^{\widetilde{BK}}(T,\varPhi,\varepsilon,x)\geq \underline{P}_{\mu}^{\widetilde{BK}}(T,\varPhi,\varepsilon)-\eta\}$$ has positive $\mu-$measure for all $\eta>0$. 
		Fix $\eta>0$ and set $s=\underline{P}_{\mu}^{\widetilde{BK}}(T,\varPhi,\varepsilon)-\eta$. For each $k\in \mathbb{N}$, set 
		$$Z_k=\{x\in Z_{\eta}:\frac{-\log \mu(B_n(x,e^{-n\varepsilon}))+\varphi_{n}(x)}{n}>s,\forall n\geq k\}.$$
		We have that $Z_k\subset Z_{k+1}$ and $\bigcup_{k=1}^{\infty}Z_k=Z_{\eta}$. Thus we can choose $k^*\in\mathbb{N}$ such that $\mu(Z_{k^*})>\frac{1}{2}\mu(Z_{\eta})>0$. Denote by $Z^*=Z_{k^*}$. For $x\in Z^*$ and $n\geq k^*$, we have
		$$\mu(B_n(x,e^{-n\varepsilon}))< \exp{[-ns+\varphi_{n}(x)]}.$$
		Fix $0<\delta<\mu(Z^*)$. For sufficiently large $k\geq k^*$ and $\mathcal{F}=\{B_{n_i}(y_i,e^{-2n_i\varepsilon})\}_{i\in I}$ a countably family with $n_i\geq k$ for each $i\in I$ and $\mu(\bigcup_{i\in I}B_{n_i}(y_i,e^{-2n_i\varepsilon}))\geq 1-\delta$. Without loss of generality, assume that $Z^*\cap B_{n_i}(y_i,e^{-2n_i\varepsilon})\not=\varnothing$ for each $i\in I$. Thus there exists $x_i\in Z^*\cap B_{n_i}(y_i,e^{-2n_i\varepsilon})$. Hence $$B_{n_i}(y_i,e^{-2n_i\varepsilon})\subset B_{n_i}(x_i,2e^{-2n_i\varepsilon})\subset B_{n_i}(x_i,e^{-n_i\varepsilon}).$$
		It follows that 
		\begin{align*}
			&\sum_{i\in I} \exp{[-n_i s+\sup_{y\in B_{n_i}(x_i,e^{-2n_i\varepsilon})}\varphi_{n_i}(y)]}\\
			\geq &\sum_{i\in I} e^{-n_i s+\varphi_{n_i}(x_i)}\\
			\geq &\sum_{i\in I}\mu(B_{n_i}(x_i,e^{-n_i\varepsilon}))\\
			\geq &\sum_{i\in I}\mu(B_{n_i}(y_i,e^{-2 n_i\varepsilon}))\\
			\geq &\mu(Z^*\cap \bigcup_{i\in I}B_{n_i}(y_i,e^{-2n_i\varepsilon}))\\
			\geq & \mu(Z^*)-\delta.
		\end{align*}
		Therefore, $$\Lambda^{s}_{k,2\varepsilon}(\mu,\varPhi,\delta)\geq \mu(Z^*)-\delta>0.$$
		It follows that $$\Lambda^{s}_{2\varepsilon}(\mu,\varPhi,\delta)=\lim_{k\rightarrow\infty}\Lambda^{s}_{k,2\varepsilon}(\mu,\varPhi,\delta)\geq \mu(Z^*)-\delta>0.$$
		By definition we have $\Lambda_{2\varepsilon}(\mu,\varPhi,\delta)\geq s$. Letting $\delta$ tend to $0$, by the arbitrariness of $\eta >0$, we have $P_{\mu}^{\tilde{K}}(T,\varPhi,2\varepsilon)\geq \underline{P}_{\mu}^{\widetilde{BK}}(T,\varPhi,\varepsilon)$. 

		Thus together with (\ref{less}), we have
		\begin{align*}
			&\lim_{\varepsilon\rightarrow 0}\sup\{\underline{P}_{\mu}^{\widetilde{BK}}(T,\varPhi,\varepsilon):\mu \in M(X),\mu(Z)=1\}\\
			\leq &\lim_{\varepsilon\rightarrow 0}\sup\{P_{\mu}^{\tilde{K}}(T,\varPhi,\varepsilon):\mu \in M(X),\mu(Z)=1\}\leq 	P^{\tilde{B}}_Z(T,\varPhi).
		\end{align*}
		Next we show that $$	M_{\varepsilon}(Z,\varPhi)\leq \sup\{\underline{P}_{\mu}^{\widetilde{BK}}(T,\varPhi,\varepsilon):\mu\in M(X),\mu(Z)=1\}.$$
		We assume that $M_{\varepsilon}(Z,\varPhi)\not=-\infty$, otherwise there is nothing to prove.  By (\ref{tempered}),  we have 
		$$\lim_{\varepsilon\rightarrow 0}\liminf_{n\rightarrow \infty}\frac{\varphi_{n}(\varepsilon)}{n}=0.$$
		For fixed $\beta>0$, we have that
		$$\liminf_{n\rightarrow \infty}\frac{\varphi_{n}(\varepsilon)}{n}>-\beta$$
		for all sufficiently small $\varepsilon>0$. By Proposition \ref{equality}, we have $P_{Z}^{\tilde{B}}(T,\varPhi)=P_{Z}^{\widetilde{WB}}(T,\varPhi)$, that is 
		$$\lim_{\varepsilon\rightarrow 0} M_{\varepsilon}(Z,\varPhi)=\lim_{\varepsilon\rightarrow 0} W_{\varepsilon}(Z,\varPhi).$$
		Fix small $\beta>0$ as above, one can choose small $\varepsilon>0$ such that 
		$$|M_{\varepsilon}(Z,\varPhi)-W_{\varepsilon}(Z,\varPhi)|<\beta.$$
		Take such an $\varepsilon>0$.  Let $s=M_{\varepsilon}(Z,\varPhi)-\beta$. Thus $W_{\varepsilon}(Z,\varPhi)>s$. By definition, we can choose $N\in \mathbb{N}$ large enough such that $c=W^s_{N,\varepsilon}(Z,\varPhi)>0$. By Lemma \ref{Frostman}, there exists $\mu\in M(X)$ with $\mu(Z)=1$ such that
		$$\mu(B_n(x,e^{-n\varepsilon}))\leq \frac{1}{c}\exp{[-ns+\sup_{y\in B_{n}(x,e^{-n\varepsilon})}\varphi_{n}(y)]}$$ 
		for any $x\in X$ and $n\geq N$. Therefore,
		$$\underline{P}_{\mu}^{\widetilde{BK}}(T,\varPhi,\varepsilon,x)\geq s+\liminf_{n\rightarrow \infty}\frac{\varphi_n(\varepsilon)}{n}\geq M_{\varepsilon}(Z,\varPhi)-2\beta$$
		for all $x\in X$. Hence 
		$$\underline{P}_{\mu}^{\widetilde{BK}}(T,\varPhi,\varepsilon)=\int \underline{P}_{\mu}^{\widetilde{BK}}(T,\varPhi,\varepsilon,x)d\mu\geq M_{\varepsilon}(Z,\varPhi)-2\beta.$$
		By the arbitrariness of  $\beta>0$,  the result follows.

	\end{proof}
\begin{Cor}
	Let $(X,T)$ be a TDS and $Z$ be a non-empty compact subset of $X$. Assume that $\varPhi=\{\varphi_n: X\rightarrow \mathbb{R}\}$ is a sequence of additive continuous potentials, that is $\varphi_n(x)=\varphi(x)+\varphi(T(x))+...+\varphi(T^{n-1}(x))$ for some continuous function $\varphi: X\rightarrow \mathbb{R}$. 
	Then 
	\begin{align*}
		P^{\tilde{B}}_Z(T,\varPhi)&=\lim_{\varepsilon\rightarrow 0}\sup\{\underline{P}_{\mu}^{\widetilde{BK}}(T,\varPhi,\varepsilon):\mu \in M(X),\, \mu(Z)=1\}\\
		&=\lim_{\varepsilon\rightarrow 0}\sup\{P_{\mu}^{\widetilde{K}}(T,\varPhi,\varepsilon):\mu \in M(X),\, \mu(Z)=1\}.
	\end{align*}
\end{Cor}
\begin{proof}
	For additive continuous potentials, the condition (\ref{tempered}) holds due to the continuity of the potentials. Thus the result follows.
\end{proof}

\end{document}